\documentclass{amsart}

\usepackage{amsmath}
\usepackage{amsfonts}
\usepackage{amssymb}
\usepackage{amsbsy}

\newtheorem{theorem}{Theorem}[section]

\newtheorem{corollary}[theorem]{Corollary}
\newtheorem{definition}[theorem]{Definition}
\newtheorem{example}[theorem]{Example}

\newtheorem{proposition}[theorem]{Proposition}
\newtheorem{remark}[theorem]{Remark}

\numberwithin{equation}{section}

\begin{document}
\title[ $\eta$- Ricci Solitons on Kenmotsu manifold with  Generalized Symmetric Metric Connection....]{$\eta$-Ricci Solitons on Kenmotsu manifold with Generalized Symmetric Metric Connection}
\author{ Mohd. Danish Siddiqi and O\u{g}uzhan Bahad\i r}
\maketitle

\begin{abstract}
The objective of the present paper is to study the $\eta$-Ricci solitons on Kenmotsu manifold with generalized symmetric metric connection of type $(\alpha,\beta)$. There are discussed Ricci and $\eta$-Ricci solitons with generalized symmetric metric connection of type $(\alpha,\beta)$ satisfying the conditions $\bar{R}.\bar{S}=0$, $\bar{S}.\bar{R}=0$, $\bar{W_{2}}.\bar{S}=0$ and $\bar{S}.\bar{W_{2}}=0.$. Finally, we construct an example of Kenmotsu manifold with generalized symmetric metric connection of type $(\alpha,\beta)$ admitting $\eta$-Ricci solitons.
\end{abstract}

 \medskip

\section{\textbf{Introduction}}
A
linear connection $\overline{\nabla}$ is said to be generalized symmetric connection if its
torsion tensor $T$ is of the form
\begin{equation}
T(X,Y)=\alpha \{u(Y)X-u(X)Y\}+\beta \{u(Y)\varphi X-u(X)\varphi Y\},  \label{Int-2}
\end{equation}
for any vector fields $X,Y$ on a manifold, where $\alpha$ and $\beta$ are smooth functions. $%
\varphi$ is a tensor of type $(1,1)$ and $u$ is a $1-$form associated with a non-vanishing smooth non-null unit vector field $\xi$.
Moreover, the connection $\overline{\nabla}$ is said to be a generalized symmetric metric connection if there is a Riemannian metric $g$ in $M$ such that $\overline{\nabla}g=0$, otherwise it is non-metric.

In the equation (\ref{Int-2}), if $\alpha=0$ ($\beta=0$), then the generalized symmetric connection is called $\beta-$ quarter-symmetric connection ( $\alpha-$ semi-symmetric connection), respectively. Moreover, if we choose $(\alpha,\beta)=(1,0)$ and $(\alpha,\beta)=(0,1)$, then the generalized symmetric connection is reduced to a semi-symmetric connection and quarter-symmetric connection, respectively. Hence a generalized symmetric connections can be viewed as a generalization of semi-symmetric connection and quarter-symmetric connection. This two connection are important for both the geometry study and applications to physics. In \cite{ref12}, H. A. Hayden introduced a metric connection with non-zero
torsion on a Riemannian manifold. The properties of Riemannian manifolds
with semi-symmetric (symmetric) and non-metric connection have been studied
by many authors (see \cite{ref1}, \cite{ref9}, \cite{ref10} , \cite{ref24}, \cite{ref26}). The idea of quarter-symmetric linear connections in a differential manifold was introduced by S.Golab \cite{ref11}.
In \cite{ref23}, Sharfuddin and Hussian defined a semi-symmetric metric connection in an
almost contact manifold, by setting
$$T(X,Y)=\eta(Y)X-\eta(X)Y.$$
In \cite{ref13}, \cite{ref25} and \cite{ref19} the authors studied the semi-symmetric metric connection and
semi-symmetric non-metric connection in a Kenmotsu manifold, respectively.

In the present paper, we defined new connection for Kenmotsu manifold, generalized symmetric metric connection. This connection is the generalized form of semi-symmetric metric connection and quarter-symmetric metric connection.

\indent
On the other hand A Ricci soliton is a natural generalization of an Einstein metric . In 1982, R. S. Hamilton \cite{ref14} said that the Ricci solitons move under the Ricci flow simply by diffeomorphisms of the initial metric, that is, they are sationary points of the Ricci flow:
\begin{equation}
\frac{\partial g}{\partial t}=-2Ric(g).
\end{equation}
\begin{definition}
A  Ricci soliton $(g,V,\lambda)$ on a Riemannian manifold is defined by
\begin{equation}
\mathcal{L}_{V}g+2S+2\lambda=0,
\end{equation}
where $S$ is the Ricci tensor, $\mathcal{L}_{V}$ is the Lie derivative along the vector field $V$ on $M$ and $\lambda$ is a real scalar. Ricci soliton is said to be shrinking, steady or expanding according as $\lambda< 0, \lambda= 0$ and $\lambda > 0$, respectively.
\end{definition}
\indent
In 1925,  H. Levy \cite{ref16} in Theorem: 4,  proved that a second order parallel symmetric non-sigular tensor in real space forms is proportional the metric tensor.
Later, R. Sharma \cite{ref22} initiated the study of Ricci solitons in contact Riemannian geometry . After that,  Tripathi \cite{ref28}, Nagaraja et. al. \cite{ref17} and others like C. S. Bagewadi et. al. \cite{ref4} extensively studied Ricci solitons in almost contact metric manifolds. In 2009, J. T. Cho and M. Kimura \cite{ref6} introduced the notion of $\eta$-Ricci solitons and gave a classification of real hypersurfaces in non-flat complex space forms admitting $\eta$-Ricci solitons. $\eta$- Ricci solitons in almost paracontact metric manifolds have been studied by A. M. Blaga et. al. \cite{ref2}.  A. M. Blaga and various others authors also have been studied $\eta$-Ricci solitons  in manifolds with different structures (see \cite{ref3}, \cite{ref20}). It is natural and interesting to study $\eta$-Ricci solitons in almost contact metric manifolds with this new connection.\\
\indent
Therefore, motivated by the above studies, in this paper we study the  $\eta$-Ricci solitons in a Kenmotsu manifold
with respect to a generalized symmetric metric connection. We shall consider $\eta$-Ricci solitons in the almost contact geometry, precisely, on an Kenmotsu manifold with  generalized symmetric metric connection which satisfies certain curvature properties: $\bar{R}.\bar{S}=0$, $\bar{S}.\bar{R}=0$, $W_{2}.\bar{S}=0$ and $\bar{S}.\bar{W_{2}}=0$ respectively.

\section{\textbf{Preliminaries}\label{sect-Kenmotsu}}

A differentiable $M$ manifold of dimension $n=2m+1$ is called almost contact metric manifold \cite{ref5}, if it admit a $(1,1)$ tensor field $\phi$, a contravaryant vector field $\xi$, a $1-$ form $\eta$ and Riemannian metric $g$ which satify
\begin{eqnarray}
\phi \xi&=&0,\\
\eta(\phi X)&=&0\\
\eta(\xi)&=&1, \label{2.1}\\
\phi^{2}(X)&=&-X+\eta(X)\xi, \label{2.2}\\
g(\phi X,\phi Y)&=&g(X,Y)-\eta(X)\eta(Y), \label{2.3}\\
g(X,\xi)&=&\eta(X), \label{2.4}
\end{eqnarray}
for all vector fields $X$, $Y$ on $M$. If we write $g(X,\phi Y)=\Phi(X,Y)$, then the tensor field $\phi$ is a anti-symmetric $(0,2)$ tensor field \cite{ref5}. If an almost contact metric manifold satisfies
\begin{eqnarray}
(\nabla_{X}\phi)Y&=&g(\phi X,Y)\xi-\eta(Y)\phi X, \label{2.9}\\
\nabla_{X}\xi&=&X-\eta(X)\xi,
\end{eqnarray}
then $M$ is called a Kenmotsu manifold, where $\nabla$ is the Levi-Civita connection of $g$ \cite{ref18}.

 In Kenmotsu manifolds the following relations hold \cite{ref18}:
\begin{eqnarray}
(\nabla_{X}\eta)Y&=&g(\phi X,\phi Y)\\
g(R(X,Y)Z,\xi)&=&\eta(R(X,Y)Z)=g(X,Z)\eta(Y)-g(Y,Z)\eta(X), \label{2.11}\\
R(\xi,X)Y&=&\eta(Y)X-g(X,Y)\xi, \label{2.12}\\
R(X,Y)\xi&=&\eta(X)Y-\eta(Y)X, \label{2.13}\\
R(\xi,X)\xi&=&X-\eta(X)\xi, \label{2.14}\\
S(X,\xi)&=&-(n-1)\eta(X), \label{2.15}\\
S(\phi X,\phi Y)&=&S(X,Y)+(n-1)\eta(X)\eta(Y) \label{2.16}
\end{eqnarray}
for any vector fields $X$, $Y$ and $Z$, where $R$ and $S$ are the the curvature and Ricci the tensors of $M$, respectively.

A Kenmotsu manifold $M$ is said to be generalized $\eta$ Einstein if its Ricci tensor S is of the form
\begin{eqnarray}
S(X,Y)=ag(X,Y)+b\eta(X)\eta(Y)+c g(\phi X,Y),
\end{eqnarray}
for any $X,Y\in\Gamma(TM)$, where $a$, $b$ and $c$ are scalar functions such that $b\neq0$ and $c\neq0$. If $c=0$ then $M$ is called $\eta$ Einstein manifold.

\section{Generalized Symmetric Metric Connection in a Kenmotsu Manifold}

Let $\overline{\nabla}$ be a linear connection and $\nabla$ be a Levi-Civita connection of an almost contact metric manifold $M$ such that

\begin{eqnarray}\label{2.19}
\overline{\nabla}_{X}Y=\nabla_{X}Y+H(X,Y),
\end{eqnarray}

for any vector field $X$ and $Y$. Where $H$ is a tensor of type $(1,2)$. For $\overline{\nabla}$ to be a generalized symmetric metric connection of $\nabla$, we have
\begin{eqnarray}\label{2.20}
H(X,Y)=\frac{1}{2}[T(X,Y)+T^{'}(X,Y)+T^{'}(Y,X)],
\end{eqnarray}

where $T$ is the torsion tensor of $\overline{\nabla}$ and
\begin{eqnarray}
g(T^{'}(X,Y),Z)=g(T(Z,X),Y).\label{2.21}
\end{eqnarray}
From (\ref{Int-2}) and (\ref{2.21}) we get
\begin{eqnarray}
T^{'}(X,Y)=\alpha \{\eta (X)Y-g(X,Y)\xi\}+\beta\{-\eta (X)\phi Y-g(\phi X,Y)\xi\}.\label{2.22}
\end{eqnarray}
Using (\ref{Int-2}), (\ref{2.20}) and (\ref{2.22}) we obtain
\begin{eqnarray}
H(X,Y)=\alpha \{\eta (Y)X-g(X,Y)\xi\}+\beta\{-\eta (X)\phi Y\}.\label{2.23}
\end{eqnarray}
\begin{corollary}
For a Kenmotsu manifold, generalized symmetric metric connection $\overline{\nabla}$ is given by
\begin{eqnarray}
\overline{\nabla}_{X}Y=\nabla_{X}Y+\alpha \{\eta (Y)X-g(X,Y)\xi\}-\beta\eta (X)\phi Y.\label{konnek}
\end{eqnarray}
\end{corollary}

If we choose  $(\alpha,\beta)=(1,0)$ and $(\alpha,\beta)=(0,1)$, generalized metric connection is reduced a semi-symmetric metric connection and quarter-symmetric metric connection as follows:
\begin{eqnarray}
\overline{\nabla}_{X}Y=\nabla_{X}Y+\eta (Y)X-g(X,Y)\xi,\label{konnek1}
\end{eqnarray}

\begin{eqnarray}\label{konnek2}
\overline{\nabla}_{X}Y=\nabla_{X}Y-\eta (X)\phi Y.
\end{eqnarray}

From (\ref{konnek}) we have the following proposition
\begin{proposition} \label{pro}
Let M be a Kenmotsu manifold with generalized metric connection. We have the following relations:
\begin{eqnarray}
(\overline{\nabla}_{X}\phi)Y&=&(\alpha+1)\{g(\phi X,Y)\xi-\eta (Y)\phi X\},\\
\overline{\nabla}_{X}\xi&=&(\alpha+1)\{X-\eta (X)\xi\},\\
(\overline{\nabla}_{X}\eta)Y&=&(\alpha+1)\{g(X,Y)-\eta (Y)\eta(X)\},
\end{eqnarray}
for any $X,Y,Z\in\Gamma(TM)$.
\end{proposition}
\section{Curvature Tensor on Kenmotsu manifold with generalized symmetric metric connection}

Let $M$ be an $n-$ dimensional Kenmotsu manifold. The curvature tensor $%
\overline{R}$ of the generalized metric connection $\overline{%
\nabla }$ on $M$ is defined by
\begin{eqnarray}
\overline{R}(X,Y)Z={\overline{\nabla}}_{X}{\overline{\nabla}}_{Y}Z-{\overline{\nabla}}_{Y}{\overline{\nabla}}_{X}Z-{\overline{\nabla}}_{[X,Y]}Z, \label{4.1}
\end{eqnarray}
Using proposition \ref{pro}, from (\ref{konnek}) and (\ref{4.1}) we have
\begin{eqnarray}\label{4.2}
\bar{R}(X,Y)Z&=&R(X,Y)Z+\{(-\alpha^{2}-2\alpha)g(Y,Z)+(\alpha^{2}+a)\eta(Y)\eta(Z)\}X\\
&+&\{(\alpha^{2}+2\alpha)g(X,Z)+(-\alpha^{2}-\alpha)\eta(X)\eta(Z)\}Y \nonumber\\
&+&\{(\alpha^{2}+\alpha)[g(Y,Z)\eta(X)-g(X,Z)\eta(Y)]\nonumber\\
&+&(\beta+\alpha\beta)[g(X,\phi Z)\eta(Y)-g(Y,\phi Z)\eta(X)]\}\xi\nonumber\\
&+&(\beta+\alpha\beta)\eta(Y)\eta(Z)\phi X-(\beta+\alpha\beta)\eta(X)\eta(Z)\phi Y\nonumber
\end{eqnarray}
where
\begin{eqnarray}
R(X,Y)Z={\nabla}_{X}{\nabla}_{Y}Z-{\nabla}_{Y}{\nabla}_{X}Z-{\nabla}_{[X,Y]}Z, \label{4.3}
\end{eqnarray}
is the curvature tensor with respect to the Levi-Civita connection $\nabla$.
Using (\ref{4.2}) and the first Bianchi identity we have
\begin{eqnarray}\label{4.4}
\bar{R}(X,Y)Z+\bar{R}(Y,Z)X+\bar{R}(Z,X)Y
\end{eqnarray}
\begin{eqnarray}\nonumber
=2(\beta+\alpha\beta)\{\eta(X)g(\phi Y,Z)+\eta(Y)g(X,\phi Z)+\eta(Z)g(Y,\phi X)\}.
\end{eqnarray}
Hence we have the following proposition
\begin{proposition}
Let $M$ be an $n-$ dimensional Kenmotsu manifold with generalized symmetric metric connection of type $(\alpha,\beta)$. If $(\alpha,\beta)=(-1,\beta)$ or $(\alpha,\beta)=(\alpha,0)$ then the first Bianchi identity of the generalized symmetric metric connection $\overline{\nabla}$ on $M$ is provided.
\end{proposition}
Using (\ref{2.11}), (\ref{2.12}), (\ref{2.13}), (\ref{2.14}) and (\ref{4.2}) we give the following proposition:
\begin{proposition}
Let $M$ be an $n-$ dimensional Kenmotsu manifold with generalized symmetric metric connection of type $(\alpha,\beta)$. Then we have the following equations:
\begin{eqnarray}\label{4.5}
\bar{R}(X,Y)\xi=(\alpha+1)\{\eta(X)Y-\eta(Y)X+\beta[\eta(Y)\phi X-\eta(X)\phi Y]\}
\end{eqnarray}
\begin{eqnarray}\label{4.6}
\quad\quad \bar{R}(\xi,X)Y=(\alpha+1)\{\eta(Y)X-g(X,Y)\xi+\beta[\eta(Y)\phi X-g(X,\phi Y)\xi]\} ,
\end{eqnarray}
\begin{eqnarray}\label{4.7}
\bar{R}(\xi,Y)\xi=(\alpha+1)\{Y-\eta(Y)\xi-\beta\phi Y\},
\end{eqnarray}
\begin{eqnarray}\label{4.8}
\eta(\bar{R}(X,Y)Z=(\alpha+1)\{\eta(Y)g(X,Z)-\eta(X)g(Y,Z)
\end{eqnarray}
\begin{eqnarray}\nonumber
\quad\quad\quad+\beta[\eta(Y)g(X,\phi Z)-\eta(X)g(Y,\phi Z)]\}
\end{eqnarray}
for any $X,Y,Z\in\Gamma(TM)$.
\end{proposition}

\noindent
\begin{example} We consider a 3-dimensional manifold $M=\{ (x,y,z) \in R^{3}: x\neq 0  \}$, where
$(x, y, z)$ are the standard coordinates in $R^{3}$. Let ${E_{1},E_{2},E_{3}}$ be a linearly independent global
frame on $M$ given by

\begin{eqnarray}
E_{1}=x\frac{\partial}{\partial z},\;E_{2}=x\frac{\partial}{\partial y},\;E_{3}=-x\frac{\partial}{\partial x}. \label{5.1}
\end{eqnarray}
Let g be the Riemannian metric defined by $$g(E_{1},E_{2})=g(E_{1},E_{3})=g(E_{2},E_{3})=0, g(E_{1},E_{1})=g(E_{2},E_{2})=g(E_{3},E_{3})=1 ,$$
Let $\eta$ be the 1-form defined by $\eta(U)=g(U,E_{3})$, for any $U\in TM$. Let $\phi$ be the $(1,1)$ tensor field defined
by $\phi E_{1}=E_{2},\phi E_{2}=-E_{1}$ and $\phi E_{3}=0$. Then, using the linearity of $\phi$ and $g$ we have $\eta(E_{3})=1,\;\phi^{2}U=-U+\eta(U)E_{3}$ and $g(\phi U,\phi W)=g(U,W)-\eta(U)\eta(W)$ for any $U,W\in TM$. Thus for $E_{3}=\xi$, $(\phi,\xi,\eta,g)$ defines an almost contact metric manifold.

Let $\nabla$ be the Levi-Civita connection with respect to the Riemannian metric $g$. Then we
have
\begin{eqnarray}
[E_{1},E_{2}]=0,\qquad [E_{1},E_{3}]=E_{1},\qquad [E_{2},E_{3}]=E_{2},
\end{eqnarray}
Using Koszul formula for the Riemannian metric $g$, we can easily calculate
\begin{eqnarray}
\nabla_{E_{1}}E_{1}=-E_{3},\qquad \nabla_{E_{1}}E_{2}=0. \qquad \nabla_{E_{1}}E_{3}=E_{1}, \nonumber \\
\nabla_{E_{2}}E_{1}=0, \qquad \nabla_{E_{2}}E_{2}=-E_{3},\qquad \nabla_{E_{2}}E_{3}=0,\\
\nabla_{E_{3}}E_{1}=0,\qquad \nabla_{E_{3}}E_{2}=0,\qquad \nabla_{E_{3}}E_{3}=0. \nonumber
\end{eqnarray}
From the above relations, it can be easily seen that

$(\nabla_{X}\phi)Y=g(\phi X,Y)\xi-\eta(Y)\phi X,\quad \nabla_{X}\xi=X-\eta(X)\xi,$ for all $E_{3}=\xi $.
Thus the manifold $M $ is a Kenmotsu manifold with the structure $(\phi,\xi,\eta,g)$.
for $\xi= E_3$. Hence the manifold $M$ under consideration is a Kenmotsu manifold of dimension three.
\end{example}

\section{\protect\vspace{0.2cm}\textbf{Ricci and $\eta$-Ricci solitons on $(M,\phi,\xi,\eta,g,)$}}

Let $(M,\phi,\xi,\eta,g,)$ be an almost contact metric manifold. Consider the equation
\begin{eqnarray}\label{5.A1}
\mathcal{L}_{\xi}g+2\bar{S}+2\lambda+2\mu\eta\otimes\eta=0,
\end{eqnarray}
where $\mathcal{L}_{\xi}$ is the Lie derivative operator along the vector field $\xi$, $\bar{S}$ is the Ricci curvature tensor field with respect to the generalized symmetric metric connection of the metric $g$, and $\lambda$ and $\mu$ are real constants. Writing $\mathcal{L}_{\xi}$ in terms of the generalized symmetric metric connection $\bar{\nabla}$, we obtain:
\begin{eqnarray}\label{5.2}
\quad\quad 2\bar{S}(X,Y)=-g(\bar{\nabla}_{X}{\xi},Y)-g(X,\bar{\nabla}_{Y}{\xi})-2\lambda g(X,Y)-2\mu\eta(X)\eta(Y),
\end{eqnarray}
for any $X,Y\in \chi(M)$.\\
\indent
The data $(g,\xi,\lambda,\mu)$ which satisfy the equation (\ref{5.1}) is said to be an $\eta$-$\textit{Ricci soliton}$ on $M$ [10]. In particular if $\mu=0$ then $(g,\xi,\lambda)$ is called $\textit{Ricci soliton}$ \cite{ref6} and it is called $\textit{shrinking}$, $\textit{steady}$ or $\textit{expanding}$, according as $\lambda$ is negative, zero or positive respectively \cite{ref6}.\\

\indent
Here is an example of $\eta$-Ricci soliton on Kenmotsu manifold with generalized symmetric metric connection.\\
\begin{example} Let $M(\phi, \xi,\eta,g)$ be the Kenmotsu manifold considered in example 4.3 .
\end{example}
\indent
Let $\bar{\nabla}$ be a generalized symmetric metric connection, we obtain:
Using the above relations, we can calculate the non-vanishing components of the curvature tensor  as follows:
\begin{eqnarray}
R(E_{1},E_{2})E_{1}=E_{2},\;R(E_{1},E_{2})E_{2}=-E_{1},\,R(E_{1},E_{3})E_{1}=E_{3}\nonumber \\
R(E_{1},E_{3})E_{3}=-E_{1},\;R(E_{2},E_{3})E_{2}=E_{3},\;R(E_{2},E_{3})E_{3}=-E_{2} \label{re}
\end{eqnarray}
From the equations (\ref{re}) we can easily calculate the non-vanishing components of the ricci tensor  as follows:\cite{sukla}
\begin{eqnarray}
S(E_{1},E_{1})=-2,\;S(E_{2},E_{2})=-2,\;S(E_{3},E_{3})=-2
\end{eqnarray}

Now, we can make similar calculations for generalized metric connection. Using (\ref{konnek}) in the above equations, we get
\begin{eqnarray}
\overline{\nabla}_{E_{1}}E_{1}=-(1+\alpha)E_{3},&  \overline{\nabla}_{E_{1}}E_{2}=0. &  \overline{\nabla}_{E_{1}}E_{3}=(1+\alpha)E_{1}, \nonumber \\
\overline{\nabla}_{E_{2}}E_{1}=0, & \overline{\nabla}_{E_{2}}E_{2}=-(1+\alpha)E_{3},& \overline{\nabla}_{E_{2}}E_{3}=\alpha E_{2},\label{na} \\
\overline{\nabla}_{E_{3}}E_{1}=-\beta E_{2},& \overline{\nabla}_{E_{3}}E_{2}=\beta E_{1}, &\overline{\nabla}_{E_{3}}E_{3}=0. \nonumber
\end{eqnarray}
From (\ref{na}), we can calculate the non-vanishing components of curvature tensor with respect to generalized metric connection as follows:

\begin{eqnarray}
\overline{R}(E_{1},E_{2})E_{1}=(1+\alpha)^{2}E_{2},\;& \overline{R}(E_{1},E_{2})E_{2}=-(1+\alpha)^{2}E_{1},\nonumber \\
\overline{R}(E_{1},E_{3})E_{1}=(1+\alpha)E_{3}      & \overline{R}(E_{1},E_{3})E_{3}=(1+\alpha)(\beta E_{2}-E_{1}),\;\nonumber  \\
\overline{R}(E_{2},E_{3})E_{2}=(1+\alpha)E_{3},\;& \overline{R}(E_{2},E_{3})E_{3}=-(1+\alpha)(-\beta E_{1}+ E_{2})\ \label{r} \\
\overline{R}(E_{3},E_{2})E_{1}=-(1+\alpha)\beta E_{3},\;& \overline{R}(E_{3},E_{1})E_{2}=(1+\alpha)\beta E_{3},\, \nonumber
. &
\end{eqnarray}

From (\ref{r}), the non-vanishing components of the ricci tensor  as follows:
\begin{eqnarray}
\overline{S}(E_{1},E_{1})=-(1+\alpha)(2+\alpha),\;& \overline{S}(E_{2},E_{2})=-(1+\alpha)(2+\alpha),\ \nonumber \\
\overline{S}(E_{3},E_{3})=-2(1+\alpha). \ \label{s}
\end{eqnarray}
From (5.2) and (5.5) we get
\begin{equation}\nonumber
2(1+\alpha)[g(e_{i},e_{i})-\eta(e_{i})\eta(e_{i})]+2\bar{S}(e_{i},e_{i})+2\lambda g(e_{i},e_{i})+2\mu\eta(e_{i})\eta(e_{i})=0
\end{equation}
 for all $i\in\left\{1,2,3\right\}$, and we have
$\lambda={(1+\alpha)^{2}}$ $( i.e.\quad \lambda > 0)$ and  $\mu=1-\alpha^{2}$, the data $(g,\xi,\lambda,\mu)$ is an $\eta$-Ricci soliton on $(M,\phi,\xi,\eta,g)$. If $\alpha =-1$ which is steady and if $\alpha\neq-1$ which is expanding.\\

\section{\protect\vspace{0.2cm}\textbf{ Parallel symmetric second order tensors and $\eta$-Ricci solitons in Kenmotsu manifolds}}

An important geometrical object in studying Ricci solitons is well known to be a symmetric $(0,2)$-tensor field which is parallel with respect to the generalized symmetric metric connection. \\
\indent
Now, let fix $h$ a symmetric tensor field of $(0,2)$-type which we suppose to be parallel with respect to generalized symmetric metric connection$\bar{\nabla}$ that is $\bar{\nabla}h=0$. Applying Ricci identity \cite{ref7}
\begin{equation}\label{6.1}
{\bar{\nabla}^{2}}h(X,Y;Z,W)-{\bar{\nabla}}^{2}h(X,Y;Z,W)=0,
\end{equation}
we obtain the relation
\begin{equation}\label{6.2}
h(\bar{R}(X,Y)Z,W)+h(Z,\bar{R}(X,Y)W)=0.
\end{equation}
Replacing $Z=W=\xi$ in (\ref{6.2}) and by using (\ref{4.5}) and by the symmetry  of $h$ follows $h(\bar{R}(X,Y)\xi,\xi)=0$ for any $X,Y\in\chi(M)$ and
\begin{equation}\label{6.3}
	(\alpha+1)\eta(X)h(Y,\xi)-(\alpha+1)\eta(Y)h(X,\xi)
	\end{equation}
	\begin{equation}\nonumber
	+(\alpha+1)\eta(X)h(\xi,Y)-(\alpha+1)\eta(Y)h(\xi,X)
\end{equation}
\begin{equation}\nonumber
+\beta\eta(Y)h(\phi X,\xi)-\beta\eta(X)h(\phi Y,\xi)+\beta\eta(Y)h(\xi,\phi X)-\beta\eta(X)h(\xi,\phi Y)=0
\end{equation}
Putting $X=\xi$ in (\ref{6.3}) and by the virtue  of (2.4), we obtain
\begin{equation}\label{6.4}
2(\alpha+1)[h(Y,\xi)-\eta(Y)h(\xi,\xi)]-2\beta h(\phi Y,\xi)=0.
\end{equation}
or
\begin{equation}\label{6.5}
2(\alpha+1)[h(Y,\xi)- g(Y,\xi)h(\xi,\xi)]-2\beta(\phi Y,\xi)=0.
\end{equation}
Suppose $(\alpha+1)\neq 0$, $\beta=0$ it results
\begin{equation}\label{6.6}
h(Y,\xi)-\eta(Y)h(\xi,\xi)=0,
\end{equation}
for any $Y\in\chi(M)$, equivalent to
\begin{equation}\label{6.7}
h(Y,\xi)- g(Y,\xi)h(\xi,\xi)=0,
\end{equation}
for any $Y\in\chi(M)$. Differentiating the equation (\ref{6.7}) covariantly with respect to the vector field $X\in\chi(M)$, we obtain
\begin{equation}\label{6.8}
h(\bar{\nabla}_{X}Y,\xi)+h(Y,\bar{\nabla}_{X}\xi)= h(\xi,\xi)[g(\bar{\nabla}_{X}Y,\xi)+g(Y,\bar{\nabla}_{X}\xi)].
\end{equation}
Using (\ref{4.5}) in (\ref{6.8}),  we obtain
\begin{equation}\label{6.9}
h(X,Y)= h(\xi,\xi)g(X,Y),
\end{equation}
for any $X,Y\in\chi(M)$.  The above equation gives the conclusion:
\begin{theorem}
Let $(M,\phi,\xi,\eta,g,)$ be a Kenmotsu manifold with generalized symmetric metric connection also with non-vanishing $\xi$-sectional curvature and endowed with a tensor field of type $(0,2)$ which is symmetric and $\phi$-skew-symmetric. If $h$ is parallel with respect to $\bar{\nabla}$, then it is a constant multiple of the metric tensor $g$.
\end{theorem}

\indent
On a Kenmotsu manifold with generalized symmetric metric connection using equation (3.10) and $\mathcal{L}_{\xi}g=2(g-\eta\otimes\eta)$, the equation (\ref{5.2}) becomes:
\begin{equation}\label{6.10}
\bar{S}(X,Y)=-(\lambda+\alpha+1)g(X,Y)+(\alpha+1-\mu)\eta(X)\eta(Y).
\end{equation}
In particular, $X=\xi$, we obtain
\begin{equation}\label{6.11}
\bar{S}(X,\xi)=-(\lambda+\mu)\eta(X).
\end{equation}
\indent
In this case, the Ricci operator $\bar{Q}$ defined by $g(\bar{Q}X,Y)=\bar{S}(X,Y)$ has the expression
\begin{equation}\label{6.12}
\bar{Q}X=-(\lambda+\alpha+1)X+(\alpha+1-\mu)\eta(X)\eta(X)\xi.
\end{equation}
\indent
Remark that on a Kenmostu manifold with generalized symmetric metric connection, the existence of an $\eta$-Ricci soliton implies that the characteristic vector field $\xi$ is an eigenvector of Ricci operator corresponding to the eigenvalue $-(\lambda+\mu)$.\\

\indent
Now we shall apply the previous results on $\eta$-Ricci solitons.
\begin{theorem}
Let $(M,\phi,\xi,\eta,g)$ be a Kenmotsu manifold with generalized symmetric metric connection. Assume that the symmetric $(0,2)$-tensor filed $h=\mathcal{L}_{\xi}g+2S+2\mu\eta\otimes\eta$ is parallel with respect to the generalized symmetric metric connection associated to $g$. Then $(g,\xi,-\frac{1}{2}h(\xi,\xi),\mu)$ yields an $\eta$-Ricci soliton.
\end{theorem}
\begin{proof}
Now, we can calculate
\begin{equation}\label{6.13}
h(\xi,\xi)=\mathcal{L}_{\xi}g(\xi,\xi)+2\bar{S}(\xi,\xi)+2\mu\eta(\xi)\eta(\xi)=-2\lambda,
\end{equation}
so $\lambda=-\frac{1}{2}h(\xi,\xi)$. From (\ref{6.9}) we conclude that $h(X,Y)=-2\lambda g(X,Y)$, for any $X,Y\in\chi(M)$. Therefore
$\mathcal{L}_{\xi}g+2S+2\mu\eta\otimes\eta=-2\lambda g$.
\end{proof}
\noindent
For $\mu=0$ follows $\mathcal{L}_{\xi}g+2S-S(\xi,\xi)g=0$ and this gives
\begin{corollary}
On a Kenmotsu manifold $(M,\phi,\xi,\eta,g)$ with generalized symmetric metric connection with property that the symmetric $(0,2)$-tensor field $h=\mathcal{L}_{\xi}g+2S$ is parallel with respect to generalized symmetric metric connection associated to $g$, the relation (\ref{5.A1}), for $\mu=0$, defines a Ricci soliton.
\end{corollary}
\indent
Conversely, we shall study the consequences of the existence of $\eta$-Ricci solitons on a Kenmotsu manifold with generalized symmetric metric connection. From (\ref{6.10}) we give the conclusion:
\begin{theorem}
If equation (\ref{5.1}) define an $\eta$-Ricci soliton on a Kenmotsu manifold $(M,\phi,\xi,\eta,g)$ with generalized symmetric metric connection, then $(M,g)$ is $\textit{quasi-Einstein}$.
\end{theorem}
\indent
Recall that the manifold is called  $\textit{quasi-Einstein}$ \cite{ref8} if the Ricci curvature tensor field $S$ is a linear combination (with real scalars $\lambda$ and $\mu$ respectively, with $\mu\neq 0$) of $g$ and the tensor product of a non-zero $1$-from $\eta$ satisfying $\eta= g(X,\xi)$, for $\xi$ a unit vector field  and respectively,  $\textit{Einstein}$ \cite{ref8} if $S$ is collinear with $g$.
\begin{theorem}
If $(\phi,\xi,\eta,g)$ is a Kenmotsu structure with generalized symmetric metric connection on $M$ and (\ref{5.1}) defines an $\eta$-Ricci soliton on $M$, then
\begin{enumerate}
	\item $Q\circ\phi=\phi\circ Q$
	\item $Q$ and $S$ are parallel along $\xi$.
\end{enumerate}
\end{theorem}
\begin{proof}
The first statement follows from a direct computation and for the second one, note that
\begin{equation}\label{6.14}
(\bar{\nabla}_{\xi}Q)X=\bar{\nabla}_{\xi}QX-Q(\bar{\nabla}_{\xi}X)
\end{equation}
and
\begin{equation}\label{6.15}
(\bar{\nabla}_{\xi}S)(X,Y)=\xi(S(X,Y))-S(\bar{\nabla}_{\xi}X,Y)-S(X,\bar{\nabla}_{\xi}Y).
\end{equation}
Replacing $Q$ and $S$ from (\ref{6.12}) and (\ref{6.11}) we get the conclusion.
\end{proof}
\indent
A particular case arise when the manifold is $\phi$-Ricci symmetric,
which means that $\phi^{2}\circ\nabla Q=0$, that fact stated  in the next theorem.
\begin{theorem}
Let $(M,\phi,\xi,\eta,g)$ be a Kenmotsu manifold with generalized symmetric metric connection. If $M$ is $\phi$-Ricci symmetric and (\ref{5.1}) defines an $\eta$-Ricci soliton on $M$, then $\mu=1$ and $(M,g)$ is $\textit{Einstein}$ manifold \cite{ref8}.
\end{theorem}
\begin{proof}
Replacing $Q$ from (\ref{6.12}) in (\ref{6.14}) and applying $\phi^{2}$ we obtain
\begin{equation}\label{6.16}\
(\alpha+1-\mu)\eta(Y)[X-\eta(X)\xi]=0,
\end{equation}
for any $X,Y\in\chi(M)$. Follows $\mu=\alpha+1$ and $S=-(\lambda+\alpha+1)g$.
\end{proof}
\begin{remark}In particular, the existence of an $\eta$-Ricci soliton on a Kenmotsu manifold with generalized symmetric metric connection which is $\textit{Ricci symmetric}$ (i.e. $\bar{\nabla}S=0$) implies that $M$ is $\textit{Einstein}$ manifold. The class of Ricci symmetric manifold represents an extension of class of Einstein manifold to which belong also the locally symmetric manifold $(i.e. ~satisfying ~\bar{\nabla}R=0)$. The condition $\bar{\nabla}S=0$ implies $\bar{R}.\bar{S}=0$ and the manifolds satisfying this condition are called $\textit{Ricci semi-symmetric}$ \cite{ref7}.
\end{remark}
\indent
In what follows we shall consider $\eta$-Ricci solitons requiring for the curvature to satisfy $\bar{R}(\xi,X).\bar{S}=0$, $\bar{S}.\bar{R}(\xi,X)=0$, $\bar{W}_{2}(\xi,X).\bar{S}=0$ and $\bar{S}.\bar{W_{2}}(\xi,X)=0$ respectively, where the $W_2$-curvature tensor field is the curvature tensor introduced by G. P.  Pokhariyal and R. S. Mishra in \cite{ref21}:
\begin{equation}\label{6.17}
W_{2}(X,Y)Z=R(X,Y)Z+\frac{1}{dim M-1}[g(X,Z)QY-g(Y,Z)QX].
\end{equation}

\section{\protect\vspace{0.2cm}\textbf{$\eta$-Ricci solitions on a Kenmotsu manifold with generalized symmetric metric connection satisfying $\bar{R}(\xi,X).\bar{S}=0$\label{sect-curvature condition}}}

Now we consider a Kenmotsu manifold with with generalized symmetric metric connection $\bar {\nabla}$ satisfying the condition
\begin{equation}\label{7.1}
\bar{S}(\bar{R}(\xi,X)Y,Z)+\bar{S}(Y,\bar{R}(\xi,X)Z)=0,
\end{equation}
for any $X,Y\in\chi(M)$.\\
\indent
Replacing the expression of $\bar{S}$ from (\ref{6.10}) and from the symmetries of $\bar{R}$ we get
\begin{equation}\label{7.2}
(\alpha+1)(\alpha+1-\mu)[\eta(Y)g(X,Z)+\eta(Z)g(X,Y)-2\eta(X)\eta(Y)\eta(Z)]=0,
\end{equation}
for any $X,Y\in\chi(M)$.\\
For $Z=\xi$ we have
\begin{equation}\label{7.3}
(\alpha+1)(\alpha+1-\mu)g(\phi X,\phi Y)=0,
\end{equation}
for any $X,Y\in\chi(M)$.\\



\noindent
Hence we can state the following theorem:
\begin{theorem}
If a Kenmotsu manifold with a generalized symmetric metric connection $\bar {\nabla}$, $(g,\xi,\lambda,\mu)$ is an $\eta$-Ricci soliton on $M$ and  satisfies $\bar{R}(\xi,X).\bar{S}=0$, then the manifold is an $\eta$-Einstein manifold.
 \end{theorem}
\indent
For $\mu=0$, we deduce:
\begin{corollary}
On a Kenmotsu manifold with a generalized symmetric metric connection satisfying $\bar{R}(\xi,X).\bar{S}=0$, there is no $\eta$-Ricci soliton with the potential vector field $\xi$.
\end{corollary}
\section{\protect\vspace{0.2cm}\textbf{$\eta$-Ricci solitons on Kenmotsu manifold with generalized symmetric metric connection satisfying $\bar{S}.\bar{R}(\xi,X)=0$\label{sect-con}}}

In this section we consider  Kenmotsu manifold with a a generalized symmetric metric connection $\bar {S}$ satisfying the condition
\begin{equation}\label{8.1}
\bar{S}(X,\bar{R}(Y,Z)W)\xi-\bar{S}(\xi,\bar{R}(Y,Z)W)X+\bar{S}(X,Y)\bar{R}(\xi,Z)W-
\end{equation}
\begin{equation}\nonumber
-\bar{S}(\xi,Y)\bar{R}(X,Z)W+\bar{S}(X,Z)\bar{R}(Y,\xi)W-\bar{S}(\xi,Z)\bar{R}(Y,X)W+
\end{equation}
\begin{equation}\nonumber
+\bar{S}(X,W)\bar{R}(Y,Z)\xi-\bar{S}(\xi,W)\bar{R}(Y,Z)X=0
\end{equation}
for any $X,Y,Z,W\in\chi(M)$.\\
\indent
Taking the inner product with $\xi$, the equation (\ref{8.1}) becomes
\begin{equation}\label{8.2}
\bar{S}(X,\bar{R}(Y,Z)W)-\bar{S}(\xi,\bar{R}(Y,Z)W)\eta(X)+\bar{S}(X,Y)\eta(\bar{R}(\xi,Z)W)-
\end{equation}
\begin{equation}\nonumber
-\bar{S}(\xi,Y)\eta(\bar{R}(X,Z)W)+\bar{S}(X,Z)\eta(\bar{R}(Y,\xi)W)-\bar{S}(\xi,Z)\eta(\bar{R}(Y,X)W)+
\end{equation}
\begin{equation}\nonumber
+\bar{S}(X,W)\eta(\bar{R}(Y,Z)\xi)-\bar{S}(\xi,W)\eta(\bar{R}(Y,Z)X)=0
\end{equation}
for any $X,Y,Z,W\in\chi(M)$.\\
\indent
For $W=\xi$, using equation (\ref{4.5}), (\ref{4.6}), (\ref{4.8}) and (\ref{6.10}) in (\ref{8.2}), we get
\begin{equation}\label{8.3}
(\alpha+1)(2\lambda+\mu+\alpha+1)[g(X,Y)\eta(Z)-g(X,Z)\eta(Y)+\beta{g(\phi X,Y)\eta(Z)-g(\phi X,Z)\eta(Y)}]
\end{equation}
for any $X,Y,Z,W\in\chi(M)$.\\
Hence we can state the following theorem:
\begin{theorem}
If $(M,\phi,\xi,\eta,g)$ is a Kenmotsu manifold with a generalized symmetric metric connection, $(g,\xi,\lambda,\mu)$ is an $\eta$-Ricci soliton on $M$ and satisfies $\bar{S}.\bar{R}(\xi,X)=0$. Then
 \begin{equation}\label{8.6}
(\alpha+1)(2\lambda+\mu+\alpha+1)=0.
\end{equation}
 \end{theorem}
\indent
For $\mu=0$ follows $\lambda=-\frac{\alpha+1}{2}$,\begin{small}($\alpha\neq-1$)\end{small}, therefore, we have the following corollary:
\begin{corollary}
On a Kenmotsu manifold with a generalized symmetric metric connection, satisfying  $\bar{S}.\bar{R}(\xi,X)=0$, the Ricci soliton defined by (\ref{5.A1}), $\mu=0$ is either shrinking or expanding.
\end{corollary}
\indent

\section{\protect\vspace{0.2cm}\textbf{$\eta$-Ricci soliton on $(\varepsilon)$-Kenmotsu manifold with a semi-symmetric metric connection satisfying  $\bar{W_{2}}(\xi,X).\bar{S}=0$\label{sect-Curvature}}}
The condition that must be satisfied by $\bar{S}$ is
\begin{equation}\label{9.1}
\bar{S}(\bar{W_{2}}(\xi,X)Y,Z)+\bar{S}(Y,\bar{W}_{2}(\xi,X)Z)=0,
\end{equation}
for any $X,Y,Z\in\chi(M)$.\\
\indent
For $X=\xi$, using (\ref{4.5}), (\ref{4.6}), (\ref{4.8}), (\ref{6.10}) and (\ref{6.17}) in (\ref{9.1}), we get
\begin{equation}\label{9.2}
\frac{(\alpha+1-\mu)(-2\mu-2\lambda+(4\alpha+4)n)}{n}\eta(Y)\eta(Z)
\end{equation}
for any $X,Y,Z\in\chi(M)$. Hence, we can state the following:
\begin{theorem}
If $(M,\phi,\xi,\eta,g)$ is an $(2n+1)$-dimensional  Kenmotsu manifold with  a generalized symmetric metric connection, $(g,\xi,\lambda,\mu)$ is an $\eta$-Ricci soliton on $M$ and $\bar{W}_{2}(\xi,X).\bar{S}=0$, then
\begin{equation}
(\alpha+1-\mu)(-2\mu-2\lambda+(4\alpha+4)n)=0.
\end{equation}
\end{theorem}
\indent
For $\mu=0$ follows that $\lambda=\frac{(4\alpha+4)n}{2}$,\begin{small}($\alpha\neq-1$)\end{small}, therefore, we have the following corollary:
\begin{corollary}
On a Kenmotsu manifold with a generalized symmetric metric connection, satisfying  $\bar{W_{2}}(\xi,X).\bar{S}=0$, the Ricci soliton defined by (\ref{5.A1}), $\mu=0$ is either shrinking or expanding.
\end{corollary}

\section{\protect\vspace{0.2cm}\textbf{$\eta$-Ricci soliton on Kenmotsu manifold with a generalized symmetric metric connection satisfying  $\bar{S}.\bar{W_{2}}(\xi,X)=0$ \label{sect-Flatness}}}
In this section we consider an $(\varepsilon)$-Kenmotsu manifold with a semi-symmetric metric connection $\bar {\nabla}$ satisfying the condition
\begin{equation}\label{10.1}
\bar{S}(X,\bar{W_{2}}(Y,Z)V)\xi-\bar{S}(\xi,\bar{W_{2}}(Y,Z)V)X+\bar{S}(X,Y)\bar{W_{2}}(\xi,Z)V-
\end{equation}
\begin{equation}\nonumber
-\bar{S}(\xi,Y)\bar{W_{2}}(X,Z)V+\bar{S}(X,Z)\bar{W_{2}}(Y,\xi)V-\bar{S}(\xi,Z)\bar{W_{2}}(Y,X)V+
\end{equation}
\begin{equation}\nonumber
+\bar{S}(X,V)\bar{W_{2}}(Y,Z)\xi-\bar{S}(\xi,V)\bar{W_{2}}(Y,Z)X=0,
\end{equation}
for any $X,Y,Z,V\in\chi(M)$.\\
\indent
Taking the inner product with $\xi$, the equation (\ref{10.1}) becomes
\begin{equation}\label{10.2}
\bar{S}(X,\bar{W_{2}}(Y,Z)V)-\bar{S}(\xi,\bar{W_{2}}(Y,Z)V)\eta(X)+\bar{S}(X,Y)\eta(\bar{W_{2}}(\xi,Z)V)-
\end{equation}
\begin{equation}\nonumber
-\bar{S}(\xi,Y)\eta(\bar{W_{2}}(X,Z)V)+\bar{S}(X,Z)\eta(\bar{W_{2}}(Y,\xi)V)-\bar{S}(\xi,Z)\eta(\bar{W_{2}}(Y,X)V)+
\end{equation}
\begin{equation}\nonumber
+\bar{S}(X,V)\eta(\bar{W_{2}}(Y,Z)\xi)-\bar{S}(\xi,V)\eta(\bar{W_{2}}(Y,Z)X)=0,
\end{equation}
for any $X,Y,Z,V\in\chi(M)$.\\
\noindent
For $X=V=\xi$, using (\ref{4.5}), (\ref{4.6}), (\ref{4.8}), (\ref{6.10}) and (\ref{6.17})  in (\ref{10.2}), we get

\begin{equation}
\{-(\alpha+1)(2\lambda+\alpha+1+\mu)+\frac{(\lambda+\alpha+1)^{2}+(\lambda+\mu)^{2}}{2n}\}\{\eta(X)\eta(Y)-g(X,Y)\}
\end{equation}\begin{equation}
+\beta(\alpha+1)(2\lambda+\alpha+1+\mu)g(\phi X,Y)=0, \nonumber
\end{equation}

for any $X,Y,Z\in\chi(M)$.
 Hence, we can state:
\begin{theorem}
If $(M,\phi,\xi,\eta,g)$ is a $(2n+1)$-dimensional  Kenmotsu manifold with generalized symmetric metric connection, $(g,\xi,\lambda,\mu)$ is an $\eta$-Ricci soliton on $M$ and $\bar{S}.\bar{W}_{2}(\xi,X)=0$, then
\begin{equation}\label{10.3}
-(\alpha+1)(2\lambda+\alpha+1+\mu)+\frac{(\lambda+\alpha+1)^{2}+(\lambda+\mu)^{2}}{2n}=0,
\end{equation}
and
\begin{equation}\label{10.4}
\beta(\alpha+1)(2\lambda+\alpha+1+\mu)=0.
\end{equation}
\end{theorem}
\indent
For $\mu=0$ we get the following corollary:

\begin{corollary}
On a Kenmotsu manifold with a generalized symmetric metric connection satisfying $\bar{S}.\bar{W_{2}}(\xi,X)=0$, the Ricci soliton defined by (\ref{5.A1}), for $\mu=0$, we have the following expressions:\\
\textbf{(i)} $-(\alpha+1)(2\lambda+\alpha+1)+\frac{(\lambda+\alpha+1)^{2}+(\lambda)^{2}}{2n}=0$ and $\beta(\alpha+1)(2\lambda+\alpha+1)=0.$\\
\textbf{(ii)} If $\alpha=-1$ or $\alpha=-2\lambda-1$ which is steady.
\end{corollary}

{\bf{Acknowledgement.}} The authors are thankful to the referee for his/her valuable comments and suggestions towards the improvement of the paper.

\noindent
Mohd Danish Siddiqi* \newline
Department of Mathematics, \newline
College of Science,\newline
Jazan University, Jazan,\newline
Kingdom of Saudi Arabia. \newline
Emails : anallintegral@gmail.com, msiddiqi@jazanu.edu.sa\newline

\noindent O\u{g}uzhan Bahad\i r $^{**}$ (Corresponding ~author)

\noindent Department of Mathematics, Faculty of Science and Letters,
Kahramanmaras Sutcu Imam University,

\noindent Kahramanmaras, TURKEY

\noindent Email: oguzbaha@gmail.com


\begin{thebibliography}{99}
\bibitem{ref1}  Agashe, N. S.,  and Chafle, M. R., A semi symetric non-metric connection in
a Riemannian manifold, Indian J. Pure Appl. Math. 23 (1992), 399-409.


\bibitem{ref2} Blaga, A. M.,  $\eta$-Ricci solitons on Lorentzian para-Sasakian manifolds, Filomat 30 (2016),
no. 2, 489-496.

\bibitem{ref3} Blaga, A. M., $\eta$-Ricci solitons on para-Kenmotsu manifolds, Balkan J. Geom. Appl. 20
(2015), 1-13.


\bibitem{ref4} Bagewadi, C. S. and Ingalahalli, G.,  Ricci Solitons in Lorentzian $\alpha$-Sasakian Manifolds, Acta Math. Acad.
Paedagog. Nyhzi. (N.S.) 28(1) (2012), 59-68.


\bibitem{ref5} Blair, D. E., Contact manifolds in Riemannian geometry, Lecture note in Mathematics, 509, Springer-Verlag Berlin-New York, 1976.

\bibitem{ref6} Cho, J. T. and Kimura, M.,  Ricci solitons and Real hypersurfaces in a complex space form,
Tohoku Math.J., 61(2009), 205-212.

\bibitem{ref7}  Chodosh, O.,  Fong,  F. T. H., Rotational symmetry of conical Kahler-Ricci solitons,
arxiv:1304.0277v2.2013,.

 \bibitem{ref8} Chaki, M. C., Maity, R. K.,  On quasi Einstein manifolds, Publ. Math. Debrecen
57 (2000), 297-306.


\bibitem{ref9}  De, U.C. and  Kamilya, D.,  Hypersurfaces of Rieamnnian manifold with
semi-symmetric non-metric connection, J. Indian Inst. Sci. 75 (1995), 707-710.


\bibitem{ref10} Friedmann, A. and Schouten, J. A., Uber die Geometric der
halbsymmetrischen Ubertragung, Math. Z. 21 (1924), 211-223.


\bibitem{ref11} Golab, S., On semi-symmetric and quarter-symmetric linear connections,
Tensor 29 (1975), 249-254.


\bibitem{ref12} Hayden, H. A., Subspaces of space with torsion, Proc. London Math. Soc. 34 (1932),
27-50.

\bibitem{ref13} Haseeb, A., Khan, M. A. and Siddiqi, M. D., Some  more  Results  On  $\varepsilon$-Kenmotsu  Manifold With  a  Semi-Symmetric metric connection, Acta Math. Univ. Comenianae, vol. LXXXV, 1(2016), 9-20.

\bibitem{ref14} Hamilton, R. S., The Ricci flow on surfaces, Mathematics and general relativity, (Santa Cruz. CA, 1986), Contemp.
Math. 71, Amer. Math. Soc., (1988), 237-262.

\bibitem{ref15} Jun, J. B., De, U. C. and Pathak, G., On Kenmotsu manifolds, J. Korean Math. Soc. 42
(2005), no. 3, 435-445.

\bibitem{ref16} Levy, H.  Symmetric tensors of the second order whose covariant derivatives vanish, Ann. Math. 27(2) (1925), 91-98.

\bibitem{ref17} Nagaraja, H.G. and C.R. Premalatha, C. R., Ricci solitons in Kenmotsu manifolds, J. Math. Anal. 3 (2) (2012), 18-24.

\bibitem{ref18} Kenmotsu, K., A class of almost contact Riemannian manifold, Tohoku Math. J., 24 (1972), 93-103.

\bibitem{ref19} Pathak, G. and De, U. C., On a semi-symmetric connection in a Kenmotsu manifold,
Bull. Calcutta Math. Soc. 94 (2002), no. 4, 319-324.

\bibitem{ref20} Prakasha, D. G. and B. S. Hadimani, $\eta$-Ricci solitons on para-Sasakian manifolds, J. Geom.,
DOI 10.1007/s00022-016-0345-z.

\bibitem{ref21} Pokhariyal, G. P., Mishra, R. S.,  The curvature tensors and their relativistic significance,
Yokohama Math. J. 18 (1970), 105-108.

\bibitem{ref22} Sharma, R.,  Certain results on $K$-contact and $(k,\mu )$-contact manifolds, J. Geom., 89(1-2) (2008), 138-147.

\bibitem{ref23} Sharfuddin, A. and Hussain, S. I., Semi-symmetric metric connections in almost contact
manifolds, Tensor (N.S.), 30(1976), 133-139.

\bibitem{ref24} Schmidt, B. G., Conditions on a connection to be a metric connection,
Commun. Math. Phys. 29 (1973), 55-59.

\bibitem{ref25} Siddiqi,  M. D., Ahmad,  M., Ojha J. P.,  CR-submanifolds of a nearly trans-hyperpolic Sasakian manifold with semi-symmetric
non metric connection. African. Diaspora J. Math. (N.S.) 2014, 17 (1), 93–105

\bibitem{ref26} Tripathi, M. M., On a semi-symmetric metric connection in a Kenmotsu manifold, J. Pure
Math. 16(1999), 67-71.

 \bibitem{ref27}  Tripathi, M. M., Nakkar, N.,  On a semi-symmetric non-metric connection in
a Kenmotsu manifold, Bull. Cal. Math. Soc. 16 (2001), no.4, 323-330.

\bibitem{ref28} Tripathi, M. M.,  Ricci solitons in contact metric manifolds, arXiv:0801.4222 [math.DG].


\bibitem{ref29} Yano, K., On semi-symmetric metric connections, Revue Roumaine De Math. Pures Appl.\newline
15(1970), 1579-1586.

\bibitem{ref30} Yano, K. and Kon, M., Structures on Manifolds, Series in Pure Math., Vol. 3, World Sci., 1984.\newline


\end{thebibliography}
\end{document}